\documentclass[11pt]{amsart}

\usepackage{palatino, mathpazo}
\usepackage[mathscr]{eucal}
\usepackage{amssymb}
\usepackage[all]{xy}
\usepackage{graphicx}
\usepackage{texdraw}
\usepackage{color}

\newtheorem{theorem}{Theorem}[section]
\newtheorem{lemma}[theorem]{Lemma}

\newtheorem{proposition}[theorem]{Proposition}

\theoremstyle{remark}
\newtheorem{remark}[theorem]{Remark}

\numberwithin{equation}{section}

\begin{document}
\title[Simple zero property of holomorphic functions]{Simple zero property of some holomorphic functions on the moduli space of tori}
\author{Zhijie Chen}
\address{Department of Mathematical Sciences, Yau Mathematical Sciences Center,
Tsinghua University, Beijing, 100084, China }
\email{zjchen@math.tsinghua.edu.cn}
\author{Ting-Jung Kuo}
\address{Taida Institute for Mathematical Sciences (TIMS), National Taiwan University,
Taipei 10617, Taiwan }
\email{tjkuo1215@gmail.com}
\author{Chang-Shou Lin}
\address{Taida Institute for Mathematical Sciences (TIMS), Center for Advanced Study in
Theoretical Sciences (CASTS), National Taiwan University, Taipei 10617, Taiwan }
\email{cslin@math.ntu.edu.tw}

\begin{abstract}
We prove that some holomorphic functions on the moduli space of tori have
only simple zeros. Instead of computing the derivative with respect to the moduli parameter $\tau$, we introduce a conceptual
proof by applying Painlev\'{e} VI\ equation. As an application of this
simple zero property, we obtain the smoothness of all the degeneracy curves of
trivial critical points for some multiple Green function.

\end{abstract}
\maketitle

\section{Introduction}

Let $\mathbb{H}:=\{ \tau|\operatorname{Im}\tau>0\}$ be the upper half plane. A
meromorphic function $f(\tau)$ defined on $\mathbb{H}$ is called to satisfy
the \emph{simple zero property} if $f(\tau)$ has only simple zeros, i.e.
\[
f^{\prime}(\tau)\not =0\text{ \ whenever \ }f(\tau)=0.
\]
For example, we consider%
\[
F_{k}(\tau):=-(\log \vartheta_{1})_{zz}\left(  \tfrac{1}{2}\omega_{k}%
;\tau \right)  ,\text{ \ }k=1,2,3,
\]
where $\vartheta_{1}(z;\tau)$ is the odd theta function defined by%
\[
\vartheta_{1}(z;\tau):=-i\sum_{n=-\infty}^{\infty}(-1)^{n}e^{(n+\tfrac{1}%
{2})^{2}\pi i\tau}e^{(2n+1)\pi iz},
\]
and $\omega_{1}:=1$, $\omega_{2}:=\tau$, $\omega_{3}:=1+\tau$. Let
$\Lambda_{\tau}=\mathbb{Z+Z}\tau$ and $\wp(z)=\wp(z|\tau)$ be the Weierstrass
elliptic function with periods $1$ and $\tau$, defined by%
\[
\wp(z|\tau):=\frac{1}{z^{2}}+\sum_{\omega \in \Lambda_{\tau}\backslash
\{0\}}\left(  \frac{1}{(z-\omega)^{2}}-\frac{1}{\omega^{2}}\right)  ,
\]
and $e_{k}(\tau):=\wp(\frac{\omega_{k}}{2}|\tau)$ for $k\in \{1,2,3\}$. Let
$\zeta(z)=\zeta(z|\tau):=-\int^{z}\wp(\xi|\tau)d\xi$ be the Weierstrass zeta
function with two quasi-periods:%
\begin{equation}
\eta_{1}(\tau)=\zeta(z+1|\tau)-\zeta(z|\tau),\text{ \ }\eta_{2}(\tau
)=\zeta(z+\tau|\tau)-\zeta(z|\tau). \label{quasi}%
\end{equation}
Then in terms of these Weierstrass functions, $F_{k}(\tau)$ can be expressed
as%
\[
F_{k}(\tau)=\eta_{1}(\tau)+e_{k}(\tau),\text{ \ }k=1,2,3.
\]
The simple zero property of $F_{k}(\tau)$ can be proved by a direct
computation; we present it in Appendix \ref{DWRT} as a nice exercise. In
\cite[Theorem 3.2]{CKLW}, Wang and the authors gave a conceptual proof by
using Painlev\'{e} VI equation. Another example is the classical Eisenstein
series of weight $1$: $\mathfrak{E}_{1}^{N}(\tau;k_{1},k_{2})$ with
characteric $(\frac{k_{1}}{N},\frac{k_{2}}{N})\in \mathbb{Q}^{2}$. The simple
zero property for $\mathfrak{E}_{1}^{N}(\tau;k_{1},k_{2})$ was first proved by
Dahmen \cite{Dahmen2} and later by Wang and the authors \cite{CKLW}.
Differently from Dahmen's proof, again our approach in \cite{CKLW} is to apply
Painlev\'{e} VI equation.

In general, given a meromorphic function $f(\tau)$, it is usually rather hard
to show whether $f(\tau)$ satisfies the simple zero property or not, due to
the non-trivial calculation of the derivative with respect to the moduli
parameter $\tau$; see Appendix \ref{DWRT} for example.

In this paper, we want to give new examples by developing our idea in
\cite{CKLW}. Recall that $g_{2}(\tau)$ and $g_{3}(\tau)$ are the coefficients
of the cubic polynomial%
\[
(\wp^{\prime}(z|\tau))^{2}=4\wp(z|\tau)^{3}-g_{2}(\tau)\wp(z|\tau)-g_{3}%
(\tau),
\]
i.e.%
\[
g_{2}(\tau)=-4(e_{1}(\tau)e_{2}(\tau)+e_{1}(\tau)e_{3}(\tau)+e_{2}(\tau
)e_{3}(\tau)),
\]%
\[
g_{3}(\tau)=4e_{1}(\tau)e_{2}(\tau)e_{3}(\tau).
\]
Given $k\in \{0,1,2,3\}$ and $C\in \mathbb{C}\cup \{ \infty \}$, we define
holomorphic functions on $\mathbb{H}$ by%
\begin{equation}
f_{0,C}(\tau):=\left \{
\begin{array}
[c]{c}%
12(C\eta_{1}(\tau)-\eta_{2}(\tau))^{2}-g_{2}(\tau)(C-\tau)^{2}\text{ if
}C\not =\infty,\\
12\eta_{1}(\tau)^{2}-g_{2}(\tau)\text{ \  \  \  \  \ if \  \  \  \  \ }C=\infty,
\end{array}
\right.  \label{func-0}%
\end{equation}%
\begin{equation}
f_{k,C}(\tau):=\left \{
\begin{array}
[c]{l}%
3e_{k}(\tau)(C\eta_{1}(\tau)-\eta_{2}(\tau))\\
+(\tfrac{g_{2}(\tau)}{2}-3e_{k}(\tau)^{2})(C-\tau)\text{ \  \ if \  \ }%
C\not =\infty,\\
3e_{k}(\tau)\eta_{1}(\tau)+\tfrac{g_{2}(\tau)}{2}-3e_{k}(\tau)^{2}\text{
\  \ if \  \ }C=\infty,
\end{array}
\right.  \text{\ }k=1,2,3. \label{func-k}%
\end{equation}
These holomorphic functions are closely related to the Hessian of trivial
critical points of some multiple Green function; we explain it later. Our first main theorem is to prove

\begin{theorem}
\label{SZ-thm1}For $k\in \{0,1,2,3\}$ and $C\in \mathbb{C}\cup \{ \infty \}$,
$f_{k,C}(\tau)$ satisfies the simple zero property.
\end{theorem}

Theorem \ref{SZ-thm1} with $C=\infty$ can be proved by computing the
expressions of $f_{k,\infty}^{\prime}(\tau)$ directly; see Appendix
\ref{DWRT}. However, it does not seem that this direct method work for the
general case $C\not =\infty$; see also Appendix \ref{DWRT} for the reason. The
main purpose of this paper is to explain the general idea how to link the
proof of this simple zero property (\emph{without computing the derivative})
with Painlev\'{e} VI equation. In fact, the full application of this
connection is to prove the following result. Define%
\[
\phi_{\pm}(\tau):=\tau-\frac{2\pi i}{\eta_{1}(\tau)\pm \sqrt{g_{2}(\tau)/12}},
\]%
\[
\phi_{k}(\tau):=\tau-\frac{6\pi ie_{k}(\tau)}{3e_{k}(\tau)\eta_{1}(\tau
)+\frac{g_{2}(\tau)}{2}-3e_{k}(\tau)^{2}}=\tau-\frac{6\pi ie_{k}(\tau
)}{f_{k,\infty}(\tau)}.
\]
Note that $\phi_{\pm}(\tau)$ are two branches of the same $2$-valued
meromorphic function whose branch point set is%
\[
\mathfrak{S}:=\left \{  \frac{ae^{\pi i/3}+b}{ce^{\pi i/3}+d}\left \vert
\begin{pmatrix}
a & b\\
c & d
\end{pmatrix}
\in SL(2,\mathbb{Z})\right.  \right \}  ,
\]
because $g_{2}(\tau)=0$ if and only if $\tau \in \mathfrak{S}$.

\begin{theorem}
\label{SZ-thm2}Functions $\phi_{\pm}(\tau)$ are both locally one-to-one from
$\mathbb{H}\backslash \mathfrak{S}$ to $\mathbb{C}\cup \{ \infty \}$. All
$\phi_{k}(\tau)$, $k=1,2,3$, are locally one-to-one from $\mathbb{H}$ to
$\mathbb{C}\cup \{ \infty \}$.
\end{theorem}

The reason we are interested in those functions comes from the following
multiple Green function $G_{2}(z_{1},z_{2}|\tau)$ defined on $(E_{\tau
}\backslash \{0\})^{2}\backslash \{(z_{1},z_{2})|z_{1}=z_{2}\}:$
\begin{equation}
G_{2}(z_{1},z_{2}|\tau):=G(z_{1}-z_{2}|\tau)-2G(z_{1}|\tau)-2G(z_{2}|\tau).
\label{511}%
\end{equation}
Here $E_{\tau}:=\mathbb{C}/\Lambda_{\tau}$ is the flat torus and $G(z|\tau)$
is the Green function on the torus $E_{\tau}$:%
\begin{equation}
\left \{
\begin{array}
[c]{l}%
-\Delta G(z|\tau)=\delta_{0}(z)-\frac{1}{\left \vert E_{\tau}\right \vert
}\text{ \ in }E_{\tau},\\
\int_{E_{\tau}}G(z|\tau)dz\wedge d\bar{z}=0,
\end{array}
\right.  \label{510}%
\end{equation}
where $\delta_{0}$ is the Dirac measure at $0$ and $\left \vert E_{\tau
}\right \vert $ is the area of the torus $E_{\tau}$. Remark that $G(\cdot
|\tau)$ is even and doubly periodic, so $\frac{\omega_{k}}{2}$, $k=1,2,3$, are
always critical points of $G(\cdot|\tau)$ and other critical points must
appear in pairs. In \cite{LW} Wang and the third author proved that
$G(\cdot|\tau)$ has either three or five critical points (depending on $\tau$).

Our motivation of studying $G_{2}$ comes from \cite{CLW,LW2}, where general multiple Green function $G_{n}$ was investigated due to its fundamental importance in the PDE theory of mean field equations. See \cite{CLW,LW2} for details.
A critical point $(a_{1},a_{2})$ of $G_{2}$ satisfies%
\begin{equation}
2\nabla G(a_{1}|\tau)=\nabla G(a_{1}-a_{2}|\tau),\text{ \ }2\nabla
G(a_{2}|\tau)=\nabla G(a_{2}-a_{1}|\tau). \label{SZ-15}%
\end{equation}
Clearly if $(a_{1},a_{2})$ is a critical point of $G_{2}$ then so does
$(a_{2},a_{1})$. Of course, we consider such two critical points to be
\emph{the same one}. We want to estimate the number of critical points of
$G_{2}$.\medskip

\noindent \textbf{Conjecture}. \textit{The number of critical points of}
$G_{2}(\cdot,\cdot|\tau)\leq9$. \textit{More precisely, the number must be one
of }$\{5,7,9\}$\textit{ depending on the moduli parameter} $\tau$.\medskip

This conjecture is still not settled yet. A critical point $(a_{1},a_{2})$ is
called a \emph{trivial critical point }if%
\[
\{a_{1},a_{2}\}=\{-a_{1},-a_{2}\} \text{ \ in \ }E_{\tau}\text{.}%
\]
It is known \cite{LW2} that $G_{2}$ has only five trivial critical points
$\{(\frac{1}{2}\omega_{i},\frac{1}{2}\omega_{j})$ $|$ $i\not =j\}$ and
$\{(q_{\pm},-q_{\pm})$ $|$ $\wp(q_{\pm})=\pm \sqrt{g_{2}/12}\}$; see Appendix
\ref{TCPGF} for a proof. Let $\{i,j,k\}=\{1,2,3\}$. In Appendix \ref{TCPGF}, we will also show that the
Hessian of $G_{2}$ at each trivial critical point is given by (the proof is borrowed from \cite{LW2})%
\begin{equation}
\det D^{2}G_{2}(\tfrac{1}{2}\omega_{i},\tfrac{1}{2}\omega_{j}|\tau
)=\frac{4|f_{k,\infty}(\tau)|^{2}}{(2\pi)^{4}\operatorname{Im}\tau
}\operatorname{Im}\phi_{k}(\tau), \label{SZ-1}%
\end{equation}%
\begin{equation}
\det D^{2}G_{2}(q_{\pm},-q_{\pm}|\tau)=\frac{3|g_{2}(\tau)|}{4\pi
^{4}\operatorname{Im}\tau}|\wp(q_{\pm})+\eta_{1}|^{2}\operatorname{Im}%
\phi_{\pm}(\tau). \label{SZ-2}%
\end{equation}
The computation of (\ref{SZ-1})-(\ref{SZ-2}), straightforward but not easy,
has nothing related to our main ideas. For the readers' convenience, we put it
in Appendix \ref{TCPGF}. In view of geometry, we want to determine those
$\tau$ such that one of trivial critical points is \emph{degenerate}, because bifurcation phenomena should happen and so
\emph{non-trivial} critical points of $G_{2}$ should appear near such $\tau$.
Define sets%
\[
C_{i,j}:=\left \{  \tau|\det D^{2}G_{2}(\tfrac{1}{2}\omega_{i},\tfrac{1}%
{2}\omega_{j}|\tau)=0\right \}  ,
\]%
\[
C_{\pm}:=\left \{  \tau|\det D^{2}G_{2}(q_{\pm},-q_{\pm}|\tau)=0\right \}  ,
\]
and curves%
\[
\tilde{C}_{\pm}:=\left \{  \tau \left \vert |\wp(q_{\pm})+\eta_{1}|^{2}%
\operatorname{Im}\phi_{\pm}(\tau)=0\right.  \right \}  .
\]
Then Theorem \ref{SZ-thm2} has the following application.

\begin{theorem}
[=Theorem \ref{SZ-thm3-copy}]\label{SZ-thm3}Recall $C_{i,j}$, $C_{\pm}$ and
$\tilde{C}_{\pm}$ defined in (\ref{III-41})-(\ref{III-44}).

\begin{itemize}
\item[(i)] For $\{i,j,k\}=\{1,2,3\}$, the degeneracy curve $C_{i,j}$ is smooth.

\item[(ii)] the set $C_{\pm}$ is a disjoint union of $\tilde{C}_{\pm}$ and
$\mathfrak{S}=\{ \tau|g_{2}(\tau)=0\}$, i.e. $C_{\pm}=\tilde{C}_{\pm}%
\sqcup \mathfrak{S}$. Furthermore, the degeneracy curve $\tilde{C}_{\pm}$ is smooth.
\end{itemize}
\end{theorem}

The numerical computation for those smooth curves is shown in Figure \ref{deg-curve-2}, which is borrowed from \cite{LW2}. Figure \ref{deg-curve-2} indicates some interesting problems for us, for example: (i) Why is each component of these degeneracy curves unbounded? (ii) Why do these degeneracy curves consist of nine smooth connected curves? We will turn back to these problems in a future work, where we will also continue our study of those functions $f_{k, C}(\tau)$. For example, we will prove $f_{0,\infty}(\tau)\neq 0$ whenever $\operatorname{Im}\tau\geq \frac{1}{2}$.

\begin{figure} \label{deg-curve-2}
\includegraphics[width=2.8in]{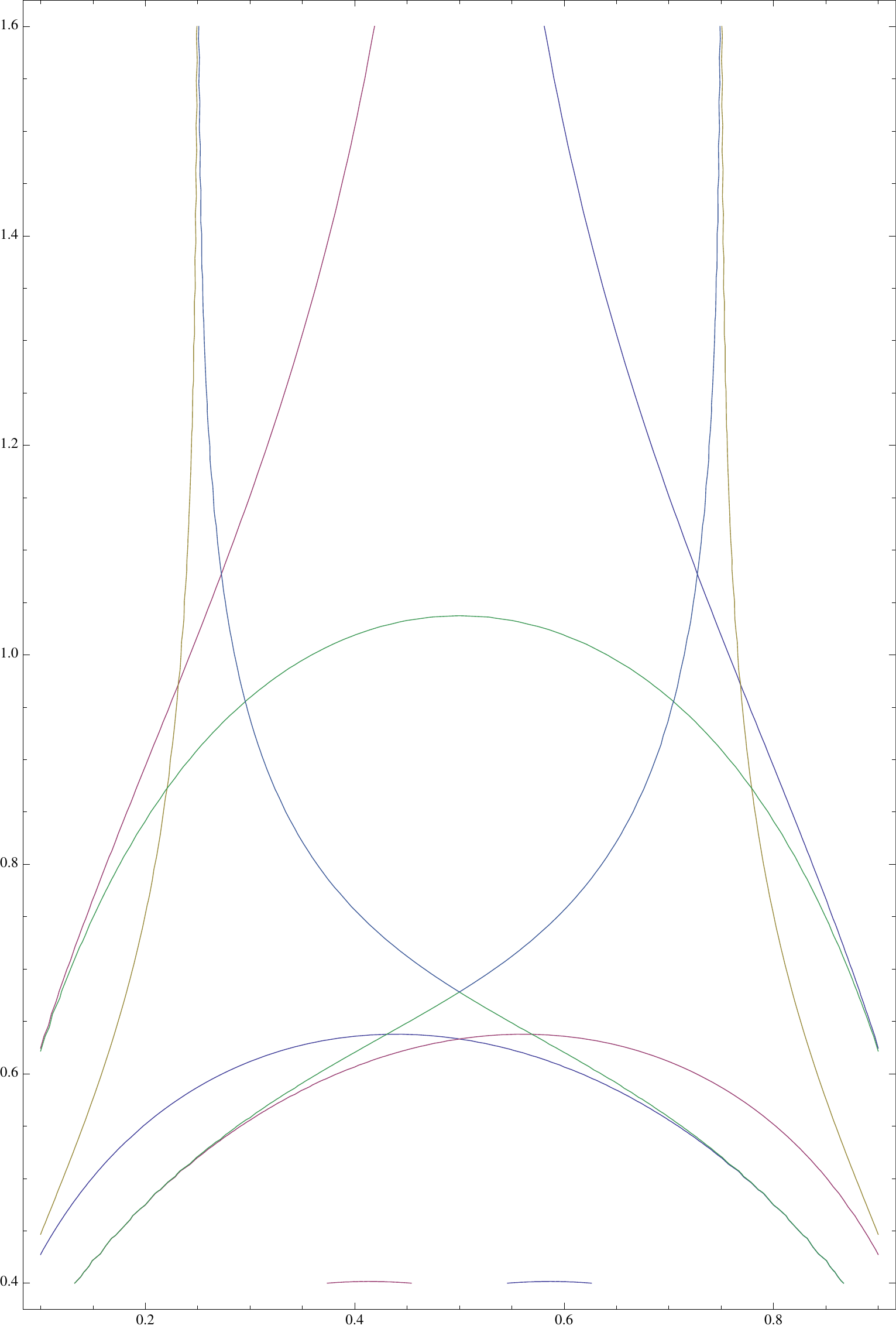}
\caption{Five degeneracy curves $C_{\pm}$ and $\tilde{C}_{i,j}$ (marked in 5 different colors).}
\end{figure}

In Section \ref{RTS-PVI}, we will review some facts about solutions of
Painlev\'{e} VI equation. Among them, there are solutions which can be reduced
to solve first order differential equations, the so-called Riccati type
equations. Riccati type equations were already well studied by Poincare in the
19th century. We could write down the explicit expression for solutions of
those Riccati type equations related to certain Painlev\'{e} VI equation. In
Section \ref{SZPA}, we will prove Theorems \ref{SZ-thm1}-\ref{SZ-thm3} by
applying the theory of Painlev\'{e} VI equation established in Section
\ref{RTS-PVI}.

\medskip

\noindent{\bf Acknowledgements} The authors thank Chin-Lung Wang very much for providing the file of Figure \ref{deg-curve-2} to us.

\section{Riccati type solutions of Painleve VI equation}

\label{RTS-PVI}

The well-known Painlev\'{e} VI equation with four free parameters
$(\alpha,\beta,\gamma,\delta)$ (PVI$(\alpha,\beta,\gamma,\delta)$) is written
as%
\begin{align}
\frac{d^{2}\lambda}{dt^{2}}=  &  \frac{1}{2}\left(  \frac{1}{\lambda}+\frac
{1}{\lambda-1}+\frac{1}{\lambda-t}\right)  \left(  \frac{d\lambda}{dt}\right)
^{2}-\left(  \frac{1}{t}+\frac{1}{t-1}+\frac{1}{\lambda-t}\right)
\frac{d\lambda}{dt}\nonumber \\
&  +\frac{\lambda(\lambda-1)(\lambda-t)}{t^{2}(t-1)^{2}}\left[  \alpha
+\beta \frac{t}{\lambda^{2}}+\gamma \frac{t-1}{(\lambda-1)^{2}}+\delta
\frac{t(t-1)}{(\lambda-t)^{2}}\right]  . \label{46-0}%
\end{align}
Due to its connection with many different disciplines in mathematics and
physics, PVI (\ref{46-0}) has been extensively studied in the past several
decades. See
\cite{Dubrovin-Mazzocco,GR,Hit1,GP,Lisovyy-Tykhyy,Y.Manin,Okamoto1,Watanabe}
and references therein.

One of the fundamental properties for PVI (\ref{46-0}) is the so-called
\emph{Painlev\'{e} property} which says that any solution $\lambda(t)$ of
(\ref{46-0}) has neither movable branch points nor movable essential
singularities; in other words, for any $t_{0}\in \mathbb{C}\backslash \{0,1\}$,
either $\lambda(t)$ is smooth at $t_{0}$ or $\lambda(t)$ has a pole at $t_{0}%
$. Furthermore, $\lambda(t)$ has at most \emph{simple poles} in $\mathbb{C}\backslash \{0,1\}$ if $\alpha
\not =0$ (cf. \cite[Proposition 1.4.1]{GP} or \cite[Theorem 3.A]{CKL}).

By the Painlev\'{e} property, it is reasonable to lift PVI (\ref{46-0}) to the
covering space $\mathbb{H=}\{ \tau|\operatorname{Im}\tau>0\}$ of
$\mathbb{C}\backslash \{0,1\}$ by the following transformation:%
\begin{equation}
t=\frac{e_{3}(\tau)-e_{1}(\tau)}{e_{2}(\tau)-e_{1}(\tau)},\text{ \ }%
\lambda(t)=\frac{\wp(p(\tau)|\tau)-e_{1}(\tau)}{e_{2}(\tau)-e_{1}(\tau)}.
\label{II-130}%
\end{equation}
Then $p(\tau)$ satisfies the following \emph{elliptic form} of PVI (cf.
\cite{Babich-Bordag,Y.Manin,Painleve}):
\begin{equation}
\frac{d^{2}p(\tau)}{d\tau^{2}}=\frac{-1}{4\pi^{2}}\sum_{k=0}^{3}\alpha_{k}%
\wp^{\prime}\left(  \left.  p(\tau)+\frac{\omega_{k}}{2}\right \vert
\tau \right)  , \label{124-0}%
\end{equation}
with parameters given by%
\begin{equation}
\left(  \alpha_{0},\alpha_{1},\alpha_{2},\alpha_{3}\right)  =\left(
\alpha,-\beta,\gamma,\tfrac{1}{2}-\delta \right)  . \label{126-0}%
\end{equation}
The Painlev\'{e} property of PVI (\ref{46-0}) implies that function
$\wp(p(\tau)|\tau)$ is a single-valued meromorphic function in $\mathbb{H}$.

Generically, solutions of PVI (\ref{46-0}) are transcendental (cf.
\cite{Watanabe}). However, there are classical solutions such as algebraic
solutions (cf. \cite{Dubrovin-Mazzocco,Lisovyy-Tykhyy}) or Riccati solutions (cf.
\cite{Watanabe}) for PVI with certain parameters. According to \cite{Watanabe}%
, when $\alpha_{k}=\tfrac{1}{2}(n_{k}+\tfrac{1}{2})^{2}$ with $n_{k}%
\in \mathbb{N\cup \{}0\mathbb{\}}$ for all $k$, PVI (\ref{46-0}) possesses four
$1$-parameter families of solutions which satisfy four Riccati type equations.
See also \cite{Mazzocco}. In this paper, we only consider the case $\alpha
_{0}=\frac{9}{8}$ and $\alpha_{k}=\frac{1}{8}$ for $k\geq1$, i.e.
PVI$(\tfrac{9}{8},\tfrac{-1}{8},\tfrac{1}{8},\tfrac{3}{8})$. We will prove
that the corresponding four Riccati type equations of PVI$(\tfrac{9}{8}%
,\tfrac{-1}{8},\tfrac{1}{8},\tfrac{3}{8})$ are%
\begin{equation}
P_{0}\left(  \tfrac{d}{dt}\lambda^{(0)}(t),\lambda^{(0)}(t),t\right)  =0,
\label{III-19}%
\end{equation}%
\begin{equation}
\frac{d\lambda^{(1)}}{dt}=\frac{3(\lambda^{(1)})^{2}-2\lambda^{(1)}%
-t}{2t(t-1)}, \label{III-20}%
\end{equation}%
\begin{equation}
\frac{d\lambda^{(2)}}{dt}=\frac{3(\lambda^{(2)})^{2}-4\lambda^{(2)}%
+t}{2t(t-1)}, \label{III-20-1}%
\end{equation}%
\begin{equation}
\frac{d\lambda^{(3)}}{dt}=\frac{3(\lambda^{(3)})^{2}-2(t+1)\lambda^{(3)}%
+t}{2t(t-1)}, \label{III-20-2}%
\end{equation}
where{\allowdisplaybreaks%
\begin{align}
P_{0}(y,x,t)  &  =8t^{3}(t-1)^{3}y^{3}-4t^{2}(t-1)^{2}(3x^{2}+2(t-2)x-t)y^{2}%
\nonumber \\
&  -2t(t-1)(x^{2}-2tx+t)(9x^{2}-2(t+4)x+t)y\label{poly-n=1}\\
&  +27x^{6}-6(7t+10)x^{5}+(4t^{2}+101t+32)x^{4}\nonumber \\
&  +4t(2t^{2}-5t-14)x^{3}-3t^{2}(4t-3)x^{2}+2t^{2}(3t+2)x-t^{3}.\nonumber
\end{align}}
\begin{remark}
Notice that equation (\ref{III-19}) is cubic in $\tfrac{d}{dt}\lambda
^{(0)}(t)$ and hence not precisely a Riccati equation. Here we call it a
Riccati type equation because it will be obtained from a Riccati equation by
applying the Okamoto transformation.
\end{remark}

The main result of this section is

\begin{theorem}
\label{thm-n=1-class}For any $k\in \{0,1,2,3\}$ and $C\in \mathbb{C}\cup \{
\infty \}$,%
\[
\lambda_{C}^{(k)}(t)=\frac{\wp(p_{C}^{(k)}(\tau)|\tau)-e_{1}(\tau)}{e_{2}%
(\tau)-e_{1}(\tau)}%
\]
is a solution of the $k$-th Riccati type equation in (\ref{III-19}%
)-(\ref{III-20-2}), where $\wp(p_{C}^{(k)}(\tau)|\tau)$ are given
by:
\begin{align}
&  \wp(p_{C}^{(0)}(\tau)|\tau)\label{III-17}\\
=  &  \frac{-4(C\eta_{1}-\eta_{2})^{3}-g_{2}(C\eta_{1}-\eta_{2})(C-\tau
)^{2}+2g_{3}(C-\tau)^{3}}{(C-\tau)[12(C\eta_{1}-\eta_{2})^{2}-g_{2}%
(C-\tau)^{2}]},\nonumber
\end{align}
and%
\begin{equation}
\wp(p_{C}^{(k)}(\tau)|\tau)=\frac{(\frac{g_{2}}{2}-3e_{k}^{2})(C\eta_{1}%
-\eta_{2})+\frac{g_{2}}{4}e_{k}(C-\tau)}{3e_{k}(C\eta_{1}-\eta_{2}%
)+(\frac{g_{2}}{2}-3e_{k}^{2})(C-\tau)} \label{III-18}%
\end{equation}
for $k\in \{1,2,3\}$.

Furthermore, such $\lambda_{C}^{(k)}(t)$ gives all the Riccati type solutions
of PVI$(\frac{9}{8},\frac{-1}{8},\frac{1}{8},\frac{3}{8})$.
\end{theorem}

We will prove Theorem \ref{thm-n=1-class} by applying the Okamoto
transformation \cite{Okamoto1} from PVI$(\frac{1}{8},\frac{-1}{8},\frac{1}%
{8},\frac{3}{8})$ to PVI$(\frac{9}{8},\frac{-1}{8},\frac{1}{8},\frac{3}{8})$.
It is well known that any solution $\lambda(t)$ of PVI (\ref{46-0})
corresponds to a solution $(\lambda(t),\mu(t))$ of a Hamiltonian system (cf.
\cite{GP}):%
\begin{equation}
\frac{d\lambda(t)}{dt}=\frac{\partial K(\lambda,\mu,t)}{\partial \mu},\text{
\ }\frac{d\mu(t)}{dt}=-\frac{\partial K(\lambda,\mu,t)}{\partial \lambda},
\label{aa}%
\end{equation}
and the Okamoto transformation is a bi-rational transformation concerning
solutions of the Hamiltonian systems (\ref{aa}) with different PVI parameters.
For our case, we note that PVI$(\tfrac{1}{2}(n+\tfrac{1}{2})^{2},\tfrac{-1}%
{8},\tfrac{1}{8},\tfrac{3}{8})$ is equivalent to the Hamiltonian system
(\ref{aa}) with $K(\lambda,\mu,t)=K_{n}$ given by%
\begin{equation}
K_{n}=\frac{1}{t(t-1)}\left \{
\begin{array}
[c]{l}%
\lambda(\lambda-1)(\lambda-t)\mu^{2}\\
-\frac{1}{2}(\lambda^{2}-2t\lambda+t)\mu-\frac{n(n+1)}{4}(\lambda-t)
\end{array}
\right \}  . \label{98}%
\end{equation}
Applying a special case of \cite[Lemma 3.3]{CKL}, we know that a solution
$(\tilde{\lambda},\tilde{\mu})$ of the Hamiltonian system (\ref{aa}%
)-(\ref{98}) corresponding to PVI$(\frac{1}{8},\frac{-1}{8},\frac{1}{8}%
,\frac{3}{8})$ (i.e. $n=0$) is transformed to a solution $(\lambda,\mu)$ of
the Hamiltonian system (\ref{aa})-(\ref{98}) corresponding to PVI$(\frac{9}%
{8},\frac{-1}{8},\frac{1}{8},\frac{3}{8})$ (i.e. $n=1$) by the following
formulas:%
\begin{equation}
\mu=\tilde{\mu}-\frac{1}{2}\left(  \frac{1}{\tilde{\lambda}}+\frac{1}%
{\tilde{\lambda}-1}+\frac{1}{\tilde{\lambda}-t}\right)  , \label{III-23}%
\end{equation}%
\begin{equation}
\lambda=\tilde{\lambda}+\frac{1}{\mu}. \label{III-24}%
\end{equation}
We recall that for PVI$(\frac{1}{8},\frac{-1}{8},\frac{1}{8},\frac{3}{8})$,
the corresponding four Riccati equations are%
\begin{equation}
\frac{d\tilde{\lambda}^{(0)}}{dt}=-\frac{(\tilde{\lambda}^{(0)})^{2}%
-2t\tilde{\lambda}^{(0)}+t}{2t(t-1)},\text{ \  \ }\tilde{\mu}^{(0)}\equiv0,
\label{1000}%
\end{equation}%
\begin{equation}
\frac{d\tilde{\lambda}^{(1)}}{dt}=\frac{(\tilde{\lambda}^{(1)})^{2}%
-2\tilde{\lambda}^{(1)}+t}{2t(t-1)},\text{ \  \ }\tilde{\mu}^{(1)}\equiv
\frac{1}{2\tilde{\lambda}^{(1)}}, \label{1001}%
\end{equation}%
\begin{equation}
\frac{d\tilde{\lambda}^{(2)}}{dt}=\frac{(\tilde{\lambda}^{(2)})^{2}%
-t}{2t(t-1)},\text{ \  \ }\tilde{\mu}^{(2)}\equiv \frac{1}{2(\tilde{\lambda
}^{(2)}-1)}, \label{1002}%
\end{equation}%
\begin{equation}
\frac{d\tilde{\lambda}^{(3)}}{dt}=\frac{(\tilde{\lambda}^{(3)})^{2}%
+2(t-1)\tilde{\lambda}^{(3)}-t}{2t(t-1)},\text{ \  \ }\tilde{\mu}^{(3)}%
\equiv \frac{1}{2(\tilde{\lambda}^{(3)}-t)}. \label{1003}%
\end{equation}
In \cite{CKL1}, we have completely solved all the equations (\ref{1000}%
)-(\ref{1003}).\medskip

\noindent \textbf{Theorem 2.A.} \cite{CKL1}\  \emph{For any }$k\in
\{0,1,2,3\}$\emph{ and }$C\in \mathbb{C}\cup \{ \infty \}$\emph{,}%
\[
\tilde{\lambda}_{C}^{(k)}(t)=\frac{\wp(\tilde{p}_{C}^{(k)}(\tau)|\tau
)-e_{1}(\tau)}{e_{2}(\tau)-e_{1}(\tau)}%
\]
\emph{is a solution of the }$k$\emph{-th Riccati equation in (\ref{1000}%
)-(\ref{1003}), where }$\wp(\tilde{p}_{C}^{(k)}(\tau)|\tau)$\emph{ are given
by:}%
\begin{equation}
\wp(\tilde{p}_{C}^{(0)}(\tau)|\tau)=\frac{\eta_{2}(\tau)-C\eta_{1}(\tau
)}{C-\tau}\text{,} \label{III-13}%
\end{equation}
\emph{and}%
\begin{equation}
\wp(\tilde{p}_{C}^{(k)}(\tau)|\tau)=\frac{e_{k}(C\eta_{1}(\tau)-\eta_{2}%
(\tau))+(\frac{g_{2}}{4}-2e_{k}^{2})(C-\tau)}{C\eta_{1}(\tau)-\eta_{2}%
(\tau)+e_{k}(C-\tau)} \label{III-14}%
\end{equation}
\emph{for} $k\in \{1,2,3\}$. \emph{Furthermore, such }$\tilde{\lambda}%
_{C}^{(k)}(t)$\emph{ gives all the Riccati type solutions of PVI}$(\frac{1}%
{8},\frac{-1}{8},\frac{1}{8},\frac{3}{8})$.\medskip

Formulae (\ref{III-13}) and (\ref{III-14}) were first obtained in
\cite{Hit1,Takemura} and later we proved in \cite{CKL1} that they actually
give solutions of the four Riccati equations (\ref{1000})-(\ref{1003}). Now we
are in a position to prove Theorem \ref{thm-n=1-class}.

\begin{proof}
[Proof of Theorem \ref{thm-n=1-class}]Let $(\tilde{\lambda},\tilde{\mu})$ be
any Riccati solution of the Hamiltonian system (\ref{aa}) corresponding to
PVI$(\frac{1}{8},\frac{-1}{8},\frac{1}{8},\frac{3}{8})$, i.e. $(\tilde
{\lambda},\tilde{\mu})$ satisfies one of the four Riccati equations
(\ref{1000})-(\ref{1003}). Let $(\lambda,\mu)$ be defined by the Okamoto
transformation (\ref{III-23})-(\ref{III-24}) via $(\tilde{\lambda},\tilde{\mu
})$. Then $(\lambda,\mu)$ is a solution of the Hamiltonian system (\ref{aa})
corresponding to PVI$(\frac{9}{8},\frac{-1}{8},\frac{1}{8},\frac{3}{8}%
)$\textit{.} Applying the Hamiltonian system (\ref{aa})-(\ref{98}) with $n=1$,
we obtain%
\begin{equation}
\lambda^{\prime}(t)=\frac{1}{t(t-1)}\left[  2\lambda(\lambda-1)(\lambda
-t)\mu-\tfrac{1}{2}(\lambda^{2}-2t\lambda+t)\right]  , \label{III-25}%
\end{equation}
which gives%
\begin{equation}
\mu=\frac{t(t-1)\lambda^{\prime}+\tfrac{1}{2}(\lambda^{2}-2t\lambda
+t)}{2\lambda(\lambda-1)(\lambda-t)}. \label{III-25-1}%
\end{equation}
In the following proof, we omit the superscript $(k)$ and the subscript $C$
for convenience, i.e. write $\tilde{\lambda}_{C}^{(k)}=\tilde{\lambda}$ and so on.

\textbf{Case 1.} $\tilde{p}(\tau)$ satisfies (\ref{III-13}), i.e.
$\tilde{\lambda}(t)$ solves the Riccati equation (\ref{1000}).

Then $\tilde{\mu}\equiv0$, so we see from (\ref{III-23})-(\ref{III-24}) that%
\[
\lambda=\tilde{\lambda}-\frac{2\tilde{\lambda}(\tilde{\lambda}-1)(\tilde
{\lambda}-t)}{\tilde{\lambda}(\tilde{\lambda}-1)+\tilde{\lambda}%
(\tilde{\lambda}-t)+(\tilde{\lambda}-1)(\tilde{\lambda}-t)}.
\]
By this, $t=\frac{e_{3}-e_{1}}{e_{2}-e_{1}}$ and $4\wp^{3}-g_{2}\wp
-g_{3}=4(\wp-e_{1})(\wp-e_{2})(\wp-e_{3})$, we easily obtain%
\begin{equation}
\wp(p(\tau)|\tau)=\frac{4\wp(\tilde{p}(\tau)|\tau)^{3}+g_{2}\wp(\tilde{p}%
(\tau)|\tau)+2g_{3}}{12\wp(\tilde{p}(\tau)|\tau)^{2}-g_{2}}. \label{III-26}%
\end{equation}
Inserting (\ref{III-13}), i.e $\wp(\tilde{p}(\tau)|\tau)=\frac{\eta_{2}%
(\tau)-C\eta_{1}(\tau)}{C-\tau}$ into (\ref{III-26}), we see that $\wp
(p(\tau)|\tau)$ satisfies (\ref{III-17}). To obtain the first order
differential equation that $\lambda(t)$ solves, we see from $\tilde{\mu}%
\equiv0$ and (\ref{III-23})-(\ref{III-24}) that%
\begin{equation}
\mu+\frac{1}{2}\left(  \frac{1}{\lambda-\frac{1}{\mu}}+\frac{1}{\lambda
-\frac{1}{\mu}-1}+\frac{1}{\lambda-\frac{1}{\mu}-t}\right)  =0. \label{III-27}%
\end{equation}
Inserting (\ref{III-25-1}) into (\ref{III-27}), a straightforward computation
gives%
\[
P_{0}(\lambda^{\prime},\lambda,t)\cdot \left[  2t(t-1)\lambda^{\prime}%
+\lambda^{2}-2t\lambda+t\right]  =0.
\]
If $2t(t-1)\lambda^{\prime}+\lambda^{2}-2t\lambda+t=0$, then $\lambda(t)$
solves the Riccati equation (\ref{1000}) and so satisfies PVI$(\frac{1}%
{8},\frac{-1}{8},\frac{1}{8},\frac{3}{8})$, clearly a contradiction. Thus
$\lambda(t)$ solves the first order algebraic equation (\ref{III-19}).

\textbf{Case 2.} $\tilde{p}(\tau)$ satisfies (\ref{III-14}), i.e.
$\tilde{\lambda}(t)$ solves the Riccati equation (\ref{1001}) if $k=1$,
(\ref{1002}) if $k=2$ and (\ref{1003}) if $k=3$.

Without loss of generality, we may assume $k=1$, since the other two cases
$k=2,3$ can be proved in the same way. Then $\tilde{\lambda}(t)$ solves the
Riccati equation (\ref{1001}) and $\tilde{\mu}\equiv \frac{1}{2\tilde{\lambda}%
}$. Consequently, (\ref{III-23})-(\ref{III-24}) give%
\begin{equation}
\lambda=\frac{(1+t)\tilde{\lambda}-2t}{2\tilde{\lambda}-1-t}, \label{III-28}%
\end{equation}
and then%
\begin{equation}
\wp(p(\tau)|\tau)=-\frac{e_{1}\wp(\tilde{p}(\tau)|\tau)+2e_{1}^{2}-\frac
{g_{2}}{2}}{2\wp(\tilde{p}(\tau)|\tau)+e_{1}}, \label{III-29}%
\end{equation}
where $t=\frac{e_{3}-e_{1}}{e_{2}-e_{1}}$, $e_{1}+e_{2}+e_{3}=0$ and
$-\frac{g_{2}}{4}=e_{2}e_{3}-e_{1}^{2}$ are used. Inserting (\ref{III-14})
with $k=1$ into (\ref{III-29}), it is easy to see that $\wp(p(\tau)|\tau)$
satisfies (\ref{III-18}) with $k=1$. To obtain the first order differential
equation that $\lambda(t)$ satisfies, we note from (\ref{III-28}) that%
\[
\tilde{\lambda}=\frac{(1+t)\lambda-2t}{2\lambda-1-t}.
\]
This, together with (\ref{III-24}), implies%
\begin{equation}
\mu=\frac{1}{\lambda-\tilde{\lambda}}=\frac{2\lambda-1-t}{2(\lambda
-1)(\lambda-t)}. \label{III-30}%
\end{equation}
Inserting (\ref{III-30}) into (\ref{III-25}), we see that $\lambda(t)$ solves
the Riccati equation (\ref{III-20}).

Finally, we note that the Okamoto transformation is invertible. Since Theorem
2.A shows that such $\tilde{\lambda}(t)$\emph{ }gives all the Riccati type
solutions of PVI$(\frac{1}{8},\frac{-1}{8},\frac{1}{8},\frac{3}{8})$, we
conclude that such $\lambda(t)$\emph{ }gives all the Riccati type solutions of
PVI$(\frac{9}{8},\frac{-1}{8},\frac{1}{8},\frac{3}{8})$. The proof is complete.
\end{proof}

Remark that if $C=\infty$, then formulae (\ref{III-17})-(\ref{III-18}) turn to
be%
\[
\wp(p_{\infty}^{(0)}(\tau)|\tau)=\frac{-4\eta_{1}^{3}-g_{2}\eta_{1}+2g_{3}%
}{12\eta_{1}^{2}-g_{2}},
\]%
\[
\wp(p_{\infty}^{(k)}(\tau)|\tau)=\frac{(\frac{g_{2}}{2}-3e_{k}^{2})\eta
_{1}+\frac{g_{2}}{4}e_{k}}{3e_{k}\eta_{1}+\frac{g_{2}}{2}-3e_{k}^{2}},
\]
respectively. To apply Theorem \ref{thm-n=1-class} in Section \ref{SZPA}, we
conclude this section by proving the following technical lemma.

\begin{lemma}
\label{lem-n=1}Let $C\in \mathbb{C}\cup \{ \infty \}$, $k\in \{1,2,3\}$ and
$\tau \in \mathbb{H}$.

\begin{itemize}
\item[(i)] If $C\not =\infty$ and%
\begin{equation}
(C-\tau)[12(C\eta_{1}-\eta_{2})^{2}-g_{2}(C-\tau)^{2}]=0, \label{III-31}%
\end{equation}
then%
\[
-4(C\eta_{1}-\eta_{2})^{3}-g_{2}(C\eta_{1}-\eta_{2})(C-\tau)^{2}+2g_{3}%
(C-\tau)^{3}\not =0.
\]

\item[(ii)] If $12\eta_{1}^{2}-g_{2}=0$, then $-4\eta_{1}^{3}-g_{2}\eta
_{1}+2g_{3}\not =0$.

\item[(iii)] If $C\not =\infty$ and%
\begin{equation}
3e_{k}(C\eta_{1}-\eta_{2})+(\tfrac{g_{2}}{2}-3e_{k}^{2})(C-\tau)=0,
\label{III-32}%
\end{equation}
then%
\[
(\tfrac{g_{2}}{2}-3e_{k}^{2})(C\eta_{1}-\eta_{2})+\tfrac{g_{2}}{4}e_{k}%
(C-\tau)\not =0.
\]

\item[(iv)] If $3e_{k}\eta_{1}+\frac{g_{2}}{2}-3e_{k}^{2}=0$, then
$(\frac{g_{2}}{2}-3e_{k}^{2})\eta_{1}+\frac{g_{2}}{4}e_{k}\not =0$.
\end{itemize}
\end{lemma}

\begin{proof}
In this proof we denote $x=C\eta_{1}-\eta_{2}$ and $y=C-\tau$ for convenience.
By the Legendre relation $\tau \eta_{1}-\eta_{2}=2\pi i$ we have $x=2\pi
i+\eta_{1}y$, so%
\begin{equation}
x,y\text{ can not vanish simultaneously.} \label{III-33}%
\end{equation}

(i)-(ii). First consider $C\not =\infty$. Suppose that both (\ref{III-31}) and%
\begin{equation}
-4x^{3}-g_{2}xy^{2}+2g_{3}y^{3}=0 \label{III-34}%
\end{equation}
hold. Then (\ref{III-33}) gives $x\not =0$, $y\not =0$ and so (\ref{III-31})
implies $x=\pm \sqrt{g_{2}/12}y$. Inserting this into (\ref{III-34}), we easily
obtain $g_{2}^{3}-27g_{3}^{2}=0$, a contradiction. This proves (i). Similarly
we can prove (ii).

(iii)-(iv). Fix $k\in \{1,2,3\}$ and let $\{i,j\}=\{1,2,3\} \backslash \{k\}$.
By $\tfrac{g_{2}}{4}=e_{k}^{2}-e_{i}e_{j}$ and $e_{i}+e_{j}+e_{k}=0$ we
have{\allowdisplaybreaks%
\begin{align*}
\det%
\begin{pmatrix}
3e_{k} & \tfrac{g_{2}}{2}-3e_{k}^{2}\\
\tfrac{g_{2}}{2}-3e_{k}^{2} & \tfrac{g_{2}}{4}e_{k}%
\end{pmatrix}
&  =(2e_{k}^{2}+e_{i}e_{j})(e_{k}^{2}-4e_{i}e_{j})\\
&  =(2e_{i}+e_{j})(e_{i}+2e_{j})(e_{i}-e_{j})^{2}\\
&  =(e_{i}-e_{k})(e_{j}-e_{k})(e_{i}-e_{j})^{2}\not =0.
\end{align*}
}Together with (\ref{III-33}), we obtain (iii)-(iv).
\end{proof}

\section{Simple zero property and application}

\label{SZPA}

The purpose of this section is to prove Theorems \ref{SZ-thm1}-\ref{SZ-thm3}.
For $C\in \mathbb{C}\cup \{ \infty \}$ and $k\in \{1,2,3\}$, we recall functions
$f_{k,C}(\tau)$ from $\mathbb{H}$ to $\mathbb{C}\cup \{ \infty \}$ defined in
Section 1:%
\begin{equation}
f_{0,C}(\tau):=\left \{
\begin{array}
[c]{c}%
12(C\eta_{1}(\tau)-\eta_{2}(\tau))^{2}-g_{2}(\tau)(C-\tau)^{2}\text{ if
}C\not =\infty,\\
12\eta_{1}(\tau)^{2}-g_{2}(\tau)\text{ \  \  \  \  \ if \  \  \  \  \ }C=\infty,
\end{array}
\right.  \label{III-35}%
\end{equation}%
\begin{equation}
f_{k,C}(\tau):=\left \{
\begin{array}
[c]{l}%
3e_{k}(\tau)(C\eta_{1}(\tau)-\eta_{2}(\tau))\\
+(\tfrac{g_{2}(\tau)}{2}-3e_{k}(\tau)^{2})(C-\tau)\text{ \  \ if \  \ }%
C\not =\infty,\\
3e_{k}(\tau)\eta_{1}(\tau)+\tfrac{g_{2}(\tau)}{2}-3e_{k}(\tau)^{2}\text{
\  \ if \  \ }C=\infty,
\end{array}
\right.  \label{III-35-1}%
\end{equation}
and functions $\phi_{\pm}(\tau)$ and $\phi_{k}(\tau)$ defined in Section 1:%
\begin{equation}
\phi_{\pm}(\tau):=\tau-\frac{2\pi i}{\eta_{1}(\tau)\pm \sqrt{g_{2}(\tau)/12}},
\label{III-35-2}%
\end{equation}%
\begin{equation}
\phi_{k}(\tau):=\tau-\frac{6\pi ie_{k}(\tau)}{3e_{k}(\tau)\eta_{1}(\tau
)+\frac{g_{2}(\tau)}{2}-3e_{k}(\tau)^{2}}=\tau-\frac{6\pi ie_{k}(\tau
)}{f_{k,\infty}(\tau)}. \label{III-36}%
\end{equation}
Clearly $\phi_{\pm}(\tau)$ have branch points at $\tau \in \mathfrak{S}$, where
\[
\mathfrak{S}=\left \{  \frac{ae^{\pi i/3}+b}{ce^{\pi i/3}+d}\left \vert
\begin{pmatrix}
a & b\\
c & d
\end{pmatrix}
\in SL(2,\mathbb{Z})\right.  \right \}  ,
\]
because $g(\tau)=0$ if and only if $\tau \in \mathfrak{S}$. Remark that%
\begin{equation}
\phi_{\pm}(\tau)=\tau-\frac{2\pi i}{\eta_{1}(\tau)}=\frac{\eta_{2}(\tau)}%
{\eta_{1}(\tau)}\not \in \mathbb{R\cup \{ \infty \}},\text{ \  \ }\forall \tau
\in \mathfrak{S.} \label{III-40}%
\end{equation}
In fact, for $\tau=\rho:=e^{\pi i/3}$, since $\wp(z|\rho)=\rho^{2}\wp(\rho
z|\rho)$ (cf. \cite{LW}), we have $\zeta(z|\rho)=\rho \zeta(\rho z|\rho)$,
which gives%
\[
\eta_{1}(\rho)=2\zeta(1/2|\rho)=2\rho \zeta(\rho/2|\rho)=\rho \eta_{2}(\rho),
\]
i.e., $\frac{\eta_{2}(\rho)}{\eta_{1}(\rho)}=\frac{1}{\rho}\not \in
\mathbb{R\cup \{ \infty \}}$. For any other $\tau \in \mathfrak{S}$, $\tau
=\frac{a\rho+b}{c\rho+d}$ for some $%
\begin{pmatrix}
a & b\\
c & d
\end{pmatrix}
\in SL(2,\mathbb{Z})$. Then by (cf. \cite[(4.3)]{CKLW})%
\[%
\begin{pmatrix}
\eta_{2}(\tau)\\
\eta_{1}(\tau)
\end{pmatrix}
=(c\rho+d)%
\begin{pmatrix}
a & b\\
c & d
\end{pmatrix}%
\begin{pmatrix}
\eta_{2}(\rho)\\
\eta_{1}(\rho)
\end{pmatrix}
,
\]
we obtain%
\[
\frac{\eta_{2}(\tau)}{\eta_{1}(\tau)}=\frac{a\eta_{2}(\rho)+b\eta_{1}(\rho
)}{c\eta_{2}(\rho)+d\eta_{1}(\rho)}=\frac{a\frac{1}{\rho}+b}{c\frac{1}{\rho
}+d}\not \in \mathbb{R\cup \{ \infty \}},
\]
so (\ref{III-40}) holds. As an application of Theorem \ref{thm-n=1-class}, we
obtain the following result, which immediately implies Theorems \ref{SZ-thm1}
and \ref{SZ-thm2}.

\begin{theorem}
\label{local-1-1}For $C\in \mathbb{C}\cup \{ \infty \}$ and $k\in \{1,2,3\}$, let
$f_{0,C}(\tau)$, $f_{k,C}(\tau)$, $\phi_{\pm}(\tau)$ and $\phi_{k}(\tau)$ be
defined in (\ref{III-35})-(\ref{III-36}).

\begin{itemize}
\item[(i)] For any $k\in \{0,1,2,3\}$, $f_{k,C}(\tau)$ has only simple zeros in
$\mathbb{H}$.

\item[(ii)] $\phi_{\pm}^{\prime}(\tau)\not =0$ for any $\tau \in \mathbb{H}%
\backslash \mathfrak{S}$. In particular, $\phi_{\pm}(\tau)$ has only simple
zeros in $\mathbb{H}$ and is locally one-to-one from $\mathbb{H}%
\backslash \mathfrak{S}$ to $\mathbb{C}\cup \{ \infty \}$.

\item[(iii)] For $k\in \{1,2,3\}$, $\phi_{k}^{\prime}(\tau)\not =0$ for any
$\tau \in \mathbb{H}$. In particular, $\phi_{k}(\tau)$ has only simple zeros in
$\mathbb{H}$ and is locally one-to-one from $\mathbb{H}$ to $\mathbb{C}\cup \{
\infty \}$.
\end{itemize}
\end{theorem}

\begin{proof}
(i) Recalling%
\[
\lambda_{C}^{(k)}(t)=\frac{\wp(p_{C}^{(k)}(\tau)|\tau)-e_{1}(\tau)}{e_{2}%
(\tau)-e_{1}(\tau)},\text{ \ }t=t(\tau)=\frac{e_{3}(\tau)-e_{1}(\tau)}%
{e_{2}(\tau)-e_{1}(\tau)},
\]
Theorem \ref{thm-n=1-class} and Lemma \ref{lem-n=1} imply that for any zero
$\tau_{0}\in \mathbb{H}$ of $f_{k,C}(\tau)$, $t(\tau_{0})\in \mathbb{C}%
\backslash \{0,1\}$ is a pole of the corresponding solution $\lambda_{C}%
^{(k)}(t)$ to PVI$(\frac{9}{8},\frac{-1}{8}$, $\frac{1}{8},\frac{3}{8})$.
Since $t^{\prime}(\tau)\not =0$ for all $\tau \in \mathbb{H}$ and any pole
$t_{0}\in \mathbb{C}\backslash \{0,1\}$ of any solution of PVI$(\frac{9}%
{8},\frac{-1}{8}$, $\frac{1}{8},\frac{3}{8})$ is a simple pole (cf.
\cite[Proposition 1.4.1]{GP} or \cite[Theorem 3.A]{CKL}), the assertion (i) follows readily.

(ii) Consider function $\phi_{+}(\tau)$. Fix any $\tau_{0}\in \mathbb{H}%
\backslash \mathfrak{S}$, then $g_{2}(\tau_{0})\not =0$. If $\eta_{1}(\tau
_{0})+\sqrt{g_{2}(\tau_{0})/12}=0$, then assertion (i) implies that $\tau_{0}$
is a simple zero of $\eta_{1}(\tau)+\sqrt{g_{2}(\tau)/12}$, so $\phi
_{+}^{\prime}(\tau_{0})\not =0$. It suffices to consider $\eta_{1}(\tau
_{0})+\sqrt{g_{2}(\tau_{0})/12}\not =0$. Then by letting%
\begin{equation}
C:=\phi_{+}(\tau_{0})=\tau_{0}-\frac{2\pi i}{\eta_{1}(\tau_{0})+\sqrt
{g_{2}(\tau_{0})/12}}, \label{III-38}%
\end{equation}
we see from $\tau \eta_{1}-\eta_{2}=2\pi i$ that $\tau_{0}$ is a zero of
\begin{align*}
\psi(\tau)  &  :=2\pi i+(C-\tau)\left(  \eta_{1}(\tau)+\sqrt{g_{2}/12}\right)
\\
&  =C\eta_{1}(\tau)-\eta_{2}(\tau)+\sqrt{g_{2}/12}(C-\tau),
\end{align*}
which implies that $\tau_{0}$ is a zero of
\[
f_{0,C}(\tau)=12\left[  C\eta_{1}(\tau)-\eta_{2}(\tau)-\sqrt{g_{2}/12}%
(C-\tau)\right]  \psi(\tau).
\]
Then (i) gives $\psi^{\prime}(\tau_{0})\not =0$, i.e.,
\[
\eta_{1}(\tau_{0})+\sqrt{g_{2}(\tau_{0})/12}+(\tau_{0}-C)\frac{d}{d\tau
}\left.  \left(  \eta_{1}(\tau)+\sqrt{g_{2}/12}\right)  \right \vert
_{\tau=\tau_{0}}\not =0.
\]
Together with (\ref{III-38}) and (\ref{III-35-2}), we easily obtain $\phi
_{+}^{\prime}(\tau_{0})\not =0$. Similarly we can prove $\phi_{-}^{\prime
}(\tau)\not =0$ for any $\tau \in \mathbb{H}\backslash \mathfrak{S}$. This together with (\ref{III-40}) proves (ii).

(iii) Now we consider function $\phi_{k}(\tau)$ defined in (\ref{III-36}). Fix
$k\in \{1,2,3\}$ and any $\tau_{0}\in \mathbb{H}$. If $f_{k,\infty}(\tau_{0}%
)=0$, then it is easy to see that $e_{k}(\tau_{0})\not =0$ and assertion (i)
shows that $\tau_{0}$ is a simple zero of $f_{k,\infty}$, so $\phi_{k}%
^{\prime}(\tau_{0})\not =0$. It suffices to consider $f_{k,\infty}(\tau
_{0})\not =0$. Then by letting%
\begin{equation}
C:=\phi_{k}(\tau_{0})=\tau_{0}-\frac{6\pi ie_{k}(\tau_{0})}{f_{k,\infty}%
(\tau_{0})}, \label{III-39}%
\end{equation}
it follows from $\tau \eta_{1}-\eta_{2}=2\pi i$ that $f_{k,C}(\tau_{0})=0$,
where $f_{k,C}(\tau)$ is defined in (\ref{III-35-1}), i.e.,%
\[
f_{k,C}(\tau)=6\pi ie_{k}(\tau)+(C-\tau)f_{k,\infty}(\tau).
\]
Again (i) gives $f_{k,C}^{\prime}(\tau_{0})\not =0$, i.e.,%
\[
6\pi ie_{k}^{\prime}(\tau_{0})-f_{k,\infty}(\tau_{0})+(C-\tau_{0})f_{k,\infty
}^{\prime}(\tau_{0})\not =0.
\]
Together with (\ref{III-39}) and (\ref{III-36}), we easily obtain $\phi
_{k}^{\prime}(\tau_{0})\not =0$.

This completes the proof.
\end{proof}

It is interesting to note that Theorem \ref{local-1-1} has important
applications in the theory of Green functions. Recall the multiple Green
function $G_{2}(z_{1},z_{2}|\tau)$ defined by
\[
G_{2}(z_{1},z_{2}|\tau):=G(z_{1}-z_{2}|\tau)-2G(z_{1}|\tau)-2G(z_{2}|\tau).
\]
It is easy to see that $G_{2}$ has five trivial critical points $\{(\frac
{1}{2}\omega_{i},\frac{1}{2}\omega_{j})$ $|$ $i\not =j\}$ and $\{(q_{\pm
},-q_{\pm})$ $|$ $\wp(q_{\pm})=\pm \sqrt{g_{2}/12}\}$, see \cite{LW2} or
Appendix \ref{TCPGF}. The function $G_{2}$ has deep connections with the
following mean field equation%
\begin{equation}
\triangle u+e^{u}=16\pi \delta_{0}\text{ \  \ in \ }E_{\tau}. \label{511-1}%
\end{equation}
It was proved in \cite{CLW} that the mean field equation (\ref{511-1}) has a solution in
$E_{\tau}$ if and only if $G_{2}(\cdot \cdot|\tau)$ has a non-trivial critical
point. Therefore, it is important to study the degeneracy of $G_{2}$ at
trivial critical points. Denote $\{i,j,k\}=\{1,2,3\}$ and note that%
\[
2e_{i}e_{j}+e_{k}^{2}-3e_{k}\eta_{1}=3e_{k}^{2}-3e_{k}\eta_{1}-\frac{g_{2}}%
{2}=-f_{k,\infty}(\tau).
\]
Recall that the Hessian of $G_{2}$ at these five trivial critical points are
expressed by (see \cite[Section 5]{LW2} or Appendix \ref{TCPGF})%
\begin{align}
&  \det D^{2}G_{2}(\tfrac{1}{2}\omega_{i},\tfrac{1}{2}\omega_{j}%
|\tau)\label{SZ-6}\\
=  &  \frac{4}{(2\pi)^{4}}\left(  |2e_{i}e_{j}+e_{k}^{2}-3e_{k}\eta_{1}%
|^{2}+\frac{2\pi}{\operatorname{Im}\tau}\operatorname{Re}\left(  3\bar{e}%
_{k}(2e_{i}e_{j}+e_{k}^{2}-3e_{k}\eta_{1})\right)  \right)  ,\nonumber
\end{align}%
\begin{align}
&  \det D^{2}G_{2}(q_{\pm},-q_{\pm}|\tau)\label{SZ-7}\\
=  &  \frac{9}{\pi^{4}}|\wp(q_{\pm})|^{2}\left(  |\wp(q_{\pm})+\eta_{1}%
|^{2}-\frac{2\pi}{\operatorname{Im}\tau}\operatorname{Re}(\wp(q_{\pm}%
)+\eta_{1})\right)  .\nonumber
\end{align}
Therefore,{\allowdisplaybreaks%
\begin{align*}
\det D^{2}G_{2}(\tfrac{1}{2}\omega_{i},\tfrac{1}{2}\omega_{j}|\tau)  &
=\frac{4}{(2\pi)^{4}}\left(  |f_{k,\infty}(\tau)|^{2}-\frac{6\pi
}{\operatorname{Im}\tau}\operatorname{Re}\left(  \bar{e}_{k}f_{k,\infty}%
(\tau)\right)  \right) \\
&  =\frac{4|f_{k,\infty}(\tau)|^{2}}{(2\pi)^{4}\operatorname{Im}\tau
}\operatorname{Im}\left(  \tau-\frac{6\pi ie_{k}(\tau)}{f_{k,\infty}(\tau
)}\right) \\
&  =\frac{4|f_{k,\infty}(\tau)|^{2}}{(2\pi)^{4}\operatorname{Im}\tau
}\operatorname{Im}\phi_{k}(\tau),
\end{align*}
}and{\allowdisplaybreaks%
\begin{align*}
&  \det D^{2}G_{2}(q_{\pm},-q_{\pm}|\tau)\\
=  &  \frac{9}{\pi^{4}\operatorname{Im}\tau}|\wp(q_{\pm})|^{2}|\wp(q_{\pm
})+\eta_{1}|^{2}\operatorname{Im}\left(  \tau-\frac{2\pi i}{\wp(q_{\pm}%
)+\eta_{1}}\right) \\
=  &  \frac{3|g_{2}(\tau)|}{4\pi^{4}\operatorname{Im}\tau}|\wp(q_{\pm}%
)+\eta_{1}|^{2}\operatorname{Im}\phi_{\pm}(\tau),
\end{align*}
}where $\phi_{\pm}(\tau)$, $\phi_{k}(\tau)$ are precisely those functions
defined in (\ref{III-35-2})-(\ref{III-36}). Thus the local one-to-one of
$\phi_{\pm}(\tau)$ and $\phi_{k}(\tau)$ in Theorem \ref{local-1-1} are
important for studying the smoothness of the curve in $\mathbb{H}$ where the
corresponding trivial critical point is a degenerate critical point of
$G_{2}(\cdot, \cdot|\tau)$. Recall the following sets defined in Section 1:%
\begin{equation}
C_{i,j}:=\left \{  \tau|\det D^{2}G_{2}(\tfrac{1}{2}\omega_{i},\tfrac{1}%
{2}\omega_{j}|\tau)=0\right \}  , \label{III-41}%
\end{equation}%
\begin{equation}
C_{\pm}:=\left \{  \tau|\det D^{2}G_{2}(q_{\pm},-q_{\pm}|\tau)=0\right \}  ,
\label{III-42}%
\end{equation}%
\begin{align}
\tilde{C}_{\pm}  &  :=\left \{  \tau \left \vert |\wp(q_{\pm})+\eta_{1}%
|^{2}\operatorname{Im}\phi_{\pm}(\tau)=0\right.  \right \} \label{III-44}\\
&  =\left \{  \tau \left \vert |\wp(q_{\pm})+\eta_{1}|^{2}-\tfrac{2\pi
}{\operatorname{Im}\tau}\operatorname{Re}(\wp(q_{\pm})+\eta_{1})=0\right.
\right \}  .\nonumber
\end{align}

\begin{theorem}
\label{SZ-thm3-copy}Recall $C_{i,j}$, $C_{\pm}$ and $\tilde{C}_{\pm}$ defined
in (\ref{III-41})-(\ref{III-44}).

\begin{itemize}
\item[(i)] For $\{i,j,k\}=\{1,2,3\}$, the curve $C_{i,j}$ is smooth.

\item[(ii)] the set $C_{\pm}$ is a disjoint union of $\tilde{C}_{\pm}$ and
$\mathfrak{S}=\{ \tau|g_{2}(\tau)=0\}$, i.e., $C_{\pm}=\tilde{C}_{\pm}%
\sqcup \mathfrak{S}$. Furthermore, the curve $\tilde{C}_{\pm}$ is smooth.
\end{itemize}
\end{theorem}

\begin{proof}
Write $\tau=a+bi$ with $a,b\in \mathbb{R}$.

(i) Fix $i,j,k$ and denote $H(\tau)=\det D^{2}G_{2}(\tfrac{1}{2}\omega
_{i},\tfrac{1}{2}\omega_{j}|\tau)$. Recall%
\[
H(\tau)=\frac{4}{(2\pi)^{4}}\left(  |f_{k,\infty}(\tau)|^{2}-\frac{6\pi}%
{b}\operatorname{Re}\left(  \bar{e}_{k}f_{k,\infty}(\tau)\right)  \right)  .
\]
For $\tau \in C_{i,j}$ such that $f_{k,\infty}(\tau)=0$, it follows from the
expression of $f_{k,\infty}=3e_{k}\eta_{1}+\frac{g_{2}}{2}-3e_{k}^{2}$ that
$e_{k}(\tau)\not =0$. Furthermore, Theorem \ref{local-1-1} says $f_{k,\infty
}^{\prime}(\tau)\not =0$, so%
\begin{align*}
\bar{e}_{k}(\tau)f_{k,\infty}^{\prime}(\tau)=  &  \left(  \operatorname{Re}%
e_{k}\operatorname{Re}f_{k,\infty}^{\prime}+\operatorname{Im}e_{k}%
\operatorname{Im}f_{k,\infty}^{\prime}\right)  (\tau)\\
&  -i\left(  \operatorname{Im}e_{k}\operatorname{Re}f_{k,\infty}^{\prime
}-\operatorname{Re}e_{k}\operatorname{Im}f_{k,\infty}^{\prime}\right)
(\tau)\not =0.
\end{align*}
Since%
\[
\frac{\partial H(\tau)}{\partial a}=-\frac{24\pi}{(2\pi)^{4}b}\left(
\operatorname{Re}e_{k}\operatorname{Re}f_{k,\infty}^{\prime}+\operatorname{Im}%
e_{k}\operatorname{Im}f_{k,\infty}^{\prime}\right)  (\tau),
\]%
\[
\frac{\partial H(\tau)}{\partial b}=-\frac{24\pi}{(2\pi)^{4}b}\left(
\operatorname{Im}e_{k}\operatorname{Re}f_{k,\infty}^{\prime}-\operatorname{Re}%
e_{k}\operatorname{Im}f_{k,\infty}^{\prime}\right)  (\tau),
\]
we obtain $\nabla H(\tau)\not =0$, namely $C_{i,j}$ is smooth near such $\tau$.

For $\tau \in C_{i,j}$ such that $f_{k,\infty}(\tau)\not =0$, we see from%
\[
H(\tau)=\frac{4|f_{k,\infty}(\tau)|^{2}}{(2\pi)^{4}b}\operatorname{Im}\phi
_{k}(\tau)
\]
that $C_{i,j}$ is defined by $\operatorname{Im}\phi_{k}=0$ near such $\tau$.
Since Theorem \ref{local-1-1} says $\phi_{k}^{\prime}(\tau)\not =0$ and%
\[
\frac{\partial \operatorname{Im}\phi_{k}}{\partial a}=\operatorname{Im}\phi
_{k}^{\prime},\text{ \  \ }\frac{\partial \operatorname{Im}\phi_{k}}{\partial
b}=\operatorname{Re}\phi_{k}^{\prime},
\]
we conclude that $C_{i,j}$ is also smooth near such $\tau$. This proves (i).

(ii) We consider the set $C_{+}$. Denote $\varphi(\tau)=\wp(q_{+})+\eta_{1}$
and
\begin{equation}
H_{+}(\tau)=|\varphi(\tau)|^{2}-\frac{2\pi}{b}\operatorname{Re}\varphi
(\tau)=\frac{1}{b}|\varphi(\tau)|^{2}\operatorname{Im}\phi_{+}(\tau).
\label{III-46}%
\end{equation}
Recalling $\wp(q_{+})=\sqrt{g_{2}/12}$, we have%
\[
\det D^{2}G_{2}(q_{+},-q_{+}|\tau)=\frac{3|g_{2}(\tau)|}{4\pi^{4}}H_{+}(\tau),
\]
which gives $C_{+}=\tilde{C}_{+}\cup \mathfrak{S}$. If $\tau \in \mathfrak{S}$,
i.e. $g_{2}(\tau)=0$, then (\ref{III-40}) shows $\operatorname{Im}\phi
_{+}(\tau)\not =0$ and $\varphi(\tau)=\eta_{1}(\tau)\not =0$, which implies
$\tau \not \in \tilde{C}_{+}$. Thus $\tilde{C}_{+}\cap \mathfrak{S}=\emptyset$,
i.e. $C_{+}=\tilde{C}_{+}\sqcup \mathfrak{S}$.

It suffices to prove that the curve $\tilde{C}_{+}$ is smooth. For those
$\tau \in \tilde{C}_{+}$ such that $\varphi(\tau)=0$, Theorem \ref{local-1-1}%
-(i) (use $f_{0,\infty}=12\eta_{1}^{2}-g_{2}$ and $g_{2}(\tau)\not =0$) implies
$\varphi^{\prime}(\tau)\not =0$. Since%
\[
\frac{\partial H_{+}(\tau)}{\partial a}=-\frac{2\pi}{b}\operatorname{Re}%
\varphi^{\prime}(\tau),\text{ \ }\frac{\partial H_{+}(\tau)}{\partial b}%
=\frac{2\pi}{b}\operatorname{Im}\varphi^{\prime}(\tau),
\]
we obtain $\nabla H_{+}(\tau)\not =0$, namely $\tilde{C}_{+}$ is smooth near
such $\tau$.

For those $\tau \in \tilde{C}_{+}$ such that $\varphi(\tau)\not =0$, we see from
(\ref{III-46}) that $\tilde{C}_{+}$ is defined by $\operatorname{Im}\phi
_{+}=0$ near such $\tau$. Since Theorem \ref{local-1-1} says $\phi_{+}%
^{\prime}(\tau)\not =0$ and%
\[
\frac{\partial \operatorname{Im}\phi_{+}}{\partial a}=\operatorname{Im}\phi
_{+}^{\prime},\text{ \  \ }\frac{\partial \operatorname{Im}\phi_{+}}{\partial
b}=\operatorname{Re}\phi_{+}^{\prime},
\]
we conclude that $\tilde{C}_{+}$ is also smooth near such $\tau$. The proof of
$C_{-}$ is similar.

The proof is complete.
\end{proof}

\appendix

\section{Trivial critical points of Green function $G_{2}$}

\label{TCPGF}

In this appendix, we prove that $G_{2}$ has five trivial critical points
$\{(\frac{1}{2}\omega_{i},\frac{1}{2}\omega_{j})$ $|$ $i\not =j\}$ and
$\{(q_{\pm},-q_{\pm})$ $|$ $\wp(q_{\pm})=\pm \sqrt{g_{2}/12}\}$ and compute the
Hessian. A critical point $(a_{1},a_{2})$ of $G_{2}$ satisfies%
\[
2\nabla G(a_{1}|\tau)=\nabla G(a_{1}-a_{2}|\tau),\text{ \ }2\nabla
G(a_{2}|\tau)=\nabla G(a_{2}-a_{1}|\tau),
\]
which is equivalent to%
\begin{equation}
\nabla G(a_{1}|\tau)+\nabla G(a_{2}|\tau)=0,\text{ \ }2\nabla G(a_{1}%
|\tau)-\nabla G(a_{1}-a_{2}|\tau)=0. \label{SZ-5}%
\end{equation}

\begin{proposition}
Green function $G_{2}(z_{1},z_{2}|\tau)$ has precisely five trivial critical
points $\{(\frac{1}{2}\omega_{i},\frac{1}{2}\omega_{j})$ $|$ $i\not =j\}$ and
$\{(q_{\pm},-q_{\pm})$ $|$ $\wp(q_{\pm})=\pm \sqrt{g_{2}/12}\}$.
\end{proposition}

\begin{proof}
Let $(a_{1},a_{2})$ be a trivial critical point, i.e.%
\[
\{a_{1},a_{2}\}=\{-a_{1},-a_{2}\} \text{ \ in \ }E_{\tau}\text{.}%
\]
If $a_{1}=-a_{1}$ and $a_{2}=-a_{2}$ in $E_{\tau}$, then $a_{k}$ are both half
periods, i.e. $(a_{1},a_{2})\in \{(\frac{1}{2}\omega_{i},\frac{1}{2}\omega
_{j})$ $|$ $i\not =j\}$. It suffices to consider the case $a_{1}=-a_{2}$ in
$E_{\tau}$. Recall from \cite{LW} that%
\begin{equation}
-4\pi G_{z}(z|\tau)=\zeta(z|\tau)-r\eta_{1}(\tau)-s\eta_{2}(\tau),
\label{SZ-8}%
\end{equation}
where $z=r+s\tau$ with $r,s\in \mathbb{R}$. Therefore, the second equation in
(\ref{SZ-5}) becomes%
\[
2\zeta(a_{1}|\tau)-\zeta(2a_{1}|\tau)=0.
\]
Together with the addition formula $2\zeta(a_{1}|\tau)-\zeta(2a_{1}%
|\tau)=\frac{-\wp^{\prime \prime}(a_{1}|\tau)}{2\wp^{\prime}(a_{1}|\tau)}$, we
obtain%
\[
6\wp(a_{1}|\tau)^{2}-g_{2}/2=\wp^{\prime \prime}(a_{1}|\tau)=0,
\]
i.e. $\wp(a_{1}|\tau)=\pm \sqrt{g_{2}/12}$. The proof is complete.
\end{proof}

Now we compute the Hessian of $G_{2}$ at these five trivial critical points.

\begin{proposition}
Let $\{i,j,k\}=\{1,2,3\}$. Then%
\begin{align}
&  \det D^{2}G_{2}(\tfrac{1}{2}\omega_{i},\tfrac{1}{2}\omega_{j}%
|\tau)\label{SZ-9}\\
=  &  \frac{4}{(2\pi)^{4}}\left(  |2e_{i}e_{j}+e_{k}^{2}-3e_{k}\eta_{1}%
|^{2}+\frac{2\pi}{b}\operatorname{Re}\left(  3\bar{e}_{k}(2e_{i}e_{j}%
+e_{k}^{2}-3e_{k}\eta_{1})\right)  \right)  .\nonumber
\end{align}%
\begin{align*}
&  \det D^{2}G_{2}(q_{\pm},-q_{\pm}|\tau)\\
&  =\frac{9}{\pi^{4}}|\wp(q_{\pm})|^{2}\left(  |\wp(q_{\pm})+\eta_{1}%
|^{2}-\frac{2\pi}{\operatorname{Im}\tau}\operatorname{Re}(\wp(q_{\pm}%
)+\eta_{1})\right)  .
\end{align*}

\end{proposition}

\begin{proof}
The following argument is borrowed from \cite[Section 5]{LW2}. First we consider
$p=(\frac{1}{2}\omega_{i},\frac{1}{2}\omega_{j})$ where $\{i,j,k\}=\{1,2,3\}$.
Notice that $\wp(\frac{1}{2}\omega_{i}-\frac{1}{2}\omega_{j})=\wp(\frac{1}%
{2}\omega_{k})=e_{k}$. For simplicity we write%
\[
\tau=a+bi\text{ and }-(e_{k}+\eta_{1})=u_{k}+v_{k}i,
\]
and similarly for the indices $i$, $j$, where $a,b,u_{k},v_{k}$ are all real.
Then by (\ref{SZ-8}), a straightforward computation gives%
\[
D^{2}G_{2}(p)=\frac{1}{2\pi}%
\begin{pmatrix}
2u_{i}-u_{k} & v_{k}-2v_{i} & u_{k} & -v_{k}\\
v_{k}-2v_{i} & u_{k}-2u_{i}-\frac{2\pi}{b} & -v_{k} & -u_{k}-\frac{2\pi}{b}\\
u_{k} & -v_{k} & 2u_{j}-u_{k} & v_{k}-2v_{j}\\
-v_{k} & -u_{k}-\frac{2\pi}{b} & v_{k}-2v_{j} & u_{k}-2u_{j}-\frac{2\pi}{b}%
\end{pmatrix}
.
\]
A lengthy yet straightforward calculation shows (\ref{SZ-9}).\ The details
will be omit here. We only note that when $\tau \in i\mathbb{R}^{+}$, all
$e_{i}$'s and $\eta_{1}$ are real numbers. Thus all the imaginary parts
vanish: $v_{1}=v_{2}=v_{3}=0$. In this case (\ref{SZ-9}) can be verified easily.

Next we consider $p=(q,-q)$ with $q\in \{q_{+},q_{-}\}$. Let $\mu=\wp(q)=u+iv$
with $u,v$ real. Since $\wp^{\prime \prime}(q)=0$, we have $\wp(2q)=-2\wp
(q)=-2\mu$ by the addition formula. Denote $\eta_{1}=s+ti$ with $s,t$ real.
Then we have%
\[
D^{2}G_{2}(p)=\frac{1}{2\pi}%
\begin{pmatrix}
-4u-s & 4v+t & 2u-s & -2v+t\\
4v+t & 4u+s-\frac{2\pi}{b} & -2v+t & -2u+s-\frac{2\pi}{b}\\
2u-s & -2v+t & -4u-s & 4v+t\\
-2v+t & -2u+s-\frac{2\pi}{b} & 4v+t & 4u+s-\frac{2\pi}{b}%
\end{pmatrix}
.
\]
A straightforward calculation easier than the previous case shows that%
\begin{align*}
\det D^{2}G_{2}(p)  &  =\frac{144}{(2\pi)^{4}}(u^{2}+v^{2})\left(
(u+s)^{2}+(v+t)^{2}-\frac{2\pi}{b}(u+s)\right) \\
&  =\frac{9}{\pi^{4}}|\wp(q)|^{2}\left(  |\wp(q)+\eta_{1}|^{2}-\frac{2\pi
}{\operatorname{Im}\tau}\operatorname{Re}(\wp(q)+\eta_{1})\right)  .
\end{align*}
The proof is complete.
\end{proof}

\section{Derivative with respect to $\tau$}

\label{DWRT}

In this appendix, as mentioned in Section 1, we prove the simple zero property
of $F_{k}(\tau)$ and $f_{k,\infty}(\tau)$ by taking derivative directly. Then
we explain why this approach can not work for $f_{k,C}(\tau)$ with
$C\not =\infty$.

We need to use the following formulas (cf. \cite{YB} or \cite[Lemma
2.3]{Chen-Kuo-Lin}):%
\[
\frac{\partial}{\partial \tau}\wp(z|\tau)=\frac{-i}{4\pi}\left[
\begin{array}
[c]{l}%
2\left(  \zeta(z|\tau)-z\eta_{1}(\tau)\right)  \wp^{\prime}(z|\tau)\\
+4\left(  \wp(z|\tau)-\eta_{1}\right)  \wp(z|\tau)-\frac{2}{3}g_{2}(\tau)
\end{array}
\right]  ,
\]%
\begin{equation}
\eta_{1}^{\prime}\left(  \tau \right)  =\frac{i}{4\pi}\left[  2\eta_{1}%
^{2}-\frac{1}{6}g_{2}(\tau)\right]  . \label{B-0}%
\end{equation}
From here, we immediately obtain%
\begin{equation}
e_{k}^{\prime}(\tau)=\frac{-i}{4\pi}\left[  4\left(  e_{k}-\eta_{1}\right)
e_{k}-\frac{2}{3}g_{2}(\tau)\right]  ,\text{ \ }k\in \{1,2,3\}. \label{B-1}%
\end{equation}
By $g_{2}=2(e_{1}^{2}+e_{2}^{2}+e_{3}^{2})$, $e_{1}+e_{2}+e_{3}=0$ and
$g_{3}=4e_{1}e_{2}e_{3}=\frac{4}{3}(e_{1}^{3}+e_{2}^{3}+e_{3}^{3})$, we easily
derive from (\ref{B-1}) that%
\begin{equation}
g_{2}^{\prime}(\tau)=\frac{-i}{\pi}[3g_{3}(\tau)-2\eta_{1}(\tau)g_{2}(\tau)],
\label{B-2}%
\end{equation}
\[
g_{3}^{\prime}(\tau)=\frac{-i}{\pi}[-3g_{3}(\tau)\eta_{1}(\tau)+\tfrac{1}{6}g_2(\tau)^{2}].
\]

\begin{proposition}
\label{P_A}For $k\in \{1,2,3\}$, $F_{k}(\tau)=\eta_{1}(\tau)+e_{k}(\tau)$
satisfies the simple zero property.
\end{proposition}

\begin{proof}
Assume $F_{k}(\tau)=0$ for some $\tau$. Then $\eta_{1}(\tau)=-e_{k}(\tau)$ and
so (\ref{B-0})-(\ref{B-1}) give%
\begin{align*}
F_{k}^{\prime}(\tau)  &  =\frac{i}{4\pi}\left[  2\eta_{1}^{2}-4\left(
e_{k}-\eta_{1}\right)  e_{k}+\frac{1}{2}g_{2}\right] \\
&  =\frac{i}{4\pi}\left[  \frac{1}{2}g_{2}-6e_{k}^{2}\right]  =\frac{-i}{4\pi
}\wp^{\prime \prime}(\tfrac{\omega_{k}}{2}|\tau)\not =0.
\end{align*}
This completes the proof.
\end{proof}

Now we consider $f_{k,\infty}(\tau)$ defined in (\ref{func-0})-(\ref{func-k}).

\begin{proposition}
\label{P_B}For $k\in \{0,1,2,3\}$, $f_{k,\infty}(\tau)$ satisfies the simple
zero property.
\end{proposition}

\begin{proof}
Recall%
\[
f_{0,\infty}(\tau)=12\eta_{1}(\tau)^{2}-g_{2}(\tau).
\]
Assume $f_{0,\infty}(\tau)=0$ for some $\tau$. Then (\ref{B-0}) implies
$\eta_{1}^{\prime}\left(  \tau \right)  =0$, so%
\[
f_{0,\infty}^{\prime}(\tau)=-g_{2}^{\prime}(\tau)=\frac{i}{\pi}[3g_{3}%
(\tau)-2\eta_{1}(\tau)g_{2}(\tau)].
\]
If $f_{0,\infty}^{\prime}(\tau)=0$, then $3g_{3}=2\eta_{1}g_{2}$, i.e.
$27g_{3}^{2}=12\eta_{1}^{2}g_{2}^{2}=g_{2}^{3}$, which is a contradiction.
Therefore, $f_{0,\infty}^{\prime}(\tau)\not =0$.

For $k\in \{1,2,3\}$,%
\[
f_{k,\infty}(\tau)=3e_{k}(\tau)\eta_{1}(\tau)+\tfrac{g_{2}(\tau)}{2}%
-3e_{k}(\tau)^{2}.
\]
By (\ref{B-0})-(\ref{B-2}), a direct computation gives%
\begin{equation}
f_{k,\infty}^{\prime}(\tau)=\frac{i}{4\pi}\left[  24e_{k}^{3}-6\eta_{1}%
(6e_{k}^{2}-3\eta_{1}e_{k}-g_{2})-\frac{9}{2}e_{k}g_{2}-6g_{3}\right]  .
\label{B-7}%
\end{equation}
Let $\{i,j\}=\{1,2,3\} \backslash \{k\}$ and denote $a=e_{k}^{2}$,
$b=-e_{i}e_{j}$. Then we easily see from $e_{i}+e_{j}+e_{k}=0$ that%
\begin{equation}
g_{2}=4(a+b),\text{ \ }g_{3}e_{k}=-4ab, \label{B-5}%
\end{equation}%
\[
a+4b=(e_{i}-e_{j})^{2}\not =0,
\]%
\[
2a-b=(e_{k}-e_{i})(e_{k}-e_{j})\not =0.
\]
Assume $f_{k,\infty}(\tau)=0$ for some $\tau$. Then{\allowdisplaybreaks
\begin{equation}
3e_{k}\eta_{1}=3e_{k}^{2}-\tfrac{g_{2}}{2}=a-2b. \label{B-6}%
\end{equation}
Inserting (\ref{B-5}) and (\ref{B-6}) into (\ref{B-7}), we obtain%
\begin{align*}
e_{k}f_{k,\infty}^{\prime}(\tau)  &  =\frac{i}{4\pi}\left[  24e_{k}%
^{4}-2(3e_{k}^{2}-\tfrac{g_{2}}{2})^{2}-\frac{9}{2}e_{k}^{2}g_{2}-6e_{k}%
g_{3}\right] \\
&  =\frac{i}{4\pi}\left[  24a^{2}-2(a-2b)^{2}-18a(a+b)+24ab\right] \\
&  =\frac{i}{2\pi}(2a-b)(a+4b)\not =0,
\end{align*}
}i.e. $f_{k,\infty}^{\prime}(\tau)\not =0$. The proof is complete.
\end{proof}

Finally, we explain why the above approach can not work for $f_{k,C}(\tau)$
with $C\not =\infty$. We take $f_{0,C}(\tau)$ for example. By $\eta_{2}%
=\tau \eta_{1}(\tau)-2\pi i$, we have%
\begin{equation}
f_{0,C}(\tau)=12[(C-\tau)\eta_{1}(\tau)+2\pi i]^{2}-g_{2}(\tau)(C-\tau)^{2}.
\label{B-8}%
\end{equation}
Then%
\begin{align}
f_{0,C}^{\prime}(\tau)  &  =24[(C-\tau)\eta_{1}(\tau)+2\pi i][(C-\tau)\eta
_{1}^{\prime}(\tau)-\eta_{1}(\tau)]\label{B-9}\\
&  -g_{2}^{\prime}(\tau)(C-\tau)^{2}+2g_{2}(\tau)(C-\tau).\nonumber
\end{align}

For function $f(\tau)=F_{k}(\tau)$ or $f_{k,\infty}(\tau)$, the key step in
the proof of Propositions \ref{P_A} and \ref{P_B} is: By $f(\tau)=0$ we can
eliminate the term $\eta_{1}$ in the expression of $f^{\prime}(\tau)$ such
that $f^{\prime}(\tau)$ turns out to be a function of $e_{1},e_{2},e_{3}$
only. Then we can use the property $e_{1}\not =e_{2}\not =e_{3}\not =e_{1}$ to
prove $f^{\prime}(\tau)\not =0$.

However, for function $f_{0,C}(\tau)$ with $C\not =\infty$, (\ref{B-9}),
(\ref{B-0}) and (\ref{B-2}) show that there are both terms $\eta_{1}$ and
$C-\tau$ in the expression of $f_{0,C}^{\prime}(\tau)$. By using $f_{0,C}%
(\tau)=0$, we \emph{can not} eliminate both $\eta_{1}$ and $C-\tau$
simultaneously in the expression of $f_{0,C}^{\prime}(\tau)$, namely we can
not express $f_{0,C}^{\prime}(\tau)$ only in terms of $e_{1},e_{2},e_{3}$.
Therefore, it seems impossible for us to prove $f_{0,C}^{\prime}(\tau)\not =0$
from its expression.


\begin{thebibliography}{99}                                                                                               %


\bibitem {Babich-Bordag}{ M. V. Babich and L. A. Bordag; \textit{The elliptic
form of the sixth Painlev\'{e} equation}. Preprint NT Z25/1997, Leipzig
(1997).}


\bibitem {YB}Y. V. Brezhnev; Non-canonical extension of $\vartheta$-functions
and modular integrability of $\vartheta$-constants. Proc. Roy. Soc. Edinburgh
Sect. A 143 (2013), no. 4, 689--738.

\bibitem {CLW}{ C.L. Chai, C.S. Lin and C.L. Wang; \textit{Mean field
equations, Hyperelliptic curves, and Modular forms: I}. Cambridge Journal of
Mathematics, \textbf{3} (2015), 127-274.}

\bibitem {Chen-Kuo-Lin}Z. Chen, T.J. Kuo and C.S. Lin; \textit{Hamiltonian
system for the elliptic form of Painlev\'{e} VI equation}. J. Math. Pures
Appl. \textbf{106} (2016), 546-581.

\bibitem {CKL}Z. Chen, T.J. Kuo and C.S. Lin; \textit{Painlev\'{e} VI
equation, modular forms and application}. preprint, 2016.

\bibitem {CKL1}Z. Chen, T.J. Kuo and C.S. Lin; \textit{Smoothness of real
solutions of Painlev\'{e} VI equation with parameter $(\frac{1}{8},\frac{-1}{8},\frac{1}{8},\frac{3}{8})$}. preprint, 2016.

\bibitem {CKLW}Z. Chen, T.J. Kuo, C.S. Lin and C.L. Wang; \textit{Green
function, Painlev\'{e} VI equation and Eisenstein series of weight one}. J.
Differ. Geom. to appear, 2016. http://www. math.ntu.edu.tw/dragon/works.htm

\bibitem {Dahmen2}S.\ Dahmen; \textit{Counting integral Lam\'{e} equations by
means of dessins d'enfants}. Trans.\ Amer.\ Math.\ Soc.\  \textbf{359} (2007), 909--922.



\bibitem {Dubrovin-Mazzocco}{ B. Dubrovin and M. Mazzocco; \textit{Monodromy
of certain Painlev\'{e}-VI transcendents and reflection groups}. Invent. Math.
\textbf{141} (2000), 55-147.}

\bibitem {GR}{ V. Gromak, I. Laine and S. Shimomura; \textit{Painlev\'{e}
differential equations in the complex plane}. de Gruyter Studies in
Mathematics 28, Berlin. New York 2002. }





\bibitem {Hit1}{ N.J. Hitchin; \textit{Twistor spaces, Einstein metrics and
isomonodromic deformations}. J. Differ. Geom. \textbf{42} (1995), no.1,
30-112. }



\bibitem {GP}{ K. Iwasaki, H. Kimura, S. Shimomura and M. Yoshida;
\textit{From Gauss to Painlev\'{e}: A Modern Theory of Special Functions}.
Springer vol. E16, 1991.}



\bibitem {LW}{ C.S. Lin and C.L. Wang; \textit{Elliptic functions, Green
functions and the mean field equations on tori}. Annals of Math. \textbf{172}
(2010), no.2, 911-954.}

\bibitem {LW2}{ C.S. Lin and C.L. Wang; \textit{Mean field equations,
Hyperelliptic curves, and Modular forms: II}. arXiv:1502.03295v2 [math. AP]
2015. }

\bibitem {Lisovyy-Tykhyy}{ O. Lisovyy and Y. Tykhyy; \textit{Algebraic
solutions of the sixth Painlev\'{e} equation}. J. Geom. Phys. \textbf{85}
(2014), 124-163. }

\bibitem {Y.Manin}{ Y. Manin; \textit{Sixth Painlev\'{e} quation, universal
elliptic curve, and mirror of} $\mathbb{P}^{2}$. Amer. Math. Soc. Transl. (2),
\textbf{186} (1998), 131--151.}

\bibitem {Mazzocco}M.\ Mazzocco; \textit{Picard and Chazy solutions to the
Painlev\'{e} VI equation}. Math.\ Ann.\  \textbf{321} (2001), 157--195.



\bibitem {Okamoto1}{ K. Okamoto; \textit{Studies on the Painlev\'{e}
equations. I. Sixth Painlev\'{e} equation} $P_{VI}$. Ann. Mat. Pura Appl.
\textbf{146} (1986), 337-381. }



\bibitem {Painleve}{ P. Painlev\'{e}; \textit{Sur les \'{e}quations
diff\'{e}rentialles du second ordre \`{a} points critiques fixes}. C. R.
Acad.\ Sic. Paris S\'{e}r. I \textbf{143} (1906), 1111-1117.}

\bibitem {Takemura}K. Takemura; \textit{The Hermite-Krichever Ansatz for
Fuchsian equations with applications to the sixth Painlev\'{e} equation and to
finite gap potentials}. Math. Z. \textbf{263} (2009), 149-194.

\bibitem {Watanabe}H. Watanabe; \textit{Birational canonical transformations
and classical solutions of the sixth Painlev\'{e} equation}. Ann. Scuola Norm.
Sup. Pisa Cl. Sci. (4), \textbf{27} (1998), 379-425.


\end{thebibliography}
\end{document}